\documentclass{birkau}

\usepackage{amsmath}
\usepackage{amssymb}
\usepackage{amsthm}
\usepackage{MnSymbol}

\numberwithin{equation}{section}

\theoremstyle{plain}
\newtheorem{theorem}{Theorem}[section]
\newtheorem{lemma}[theorem]{Lemma}

\newtheorem{corollary}[theorem]{Corollary}

\theoremstyle{definition}
\newtheorem{definition}[theorem]{Definition}

\newtheorem{remark}[theorem]{Remark}

\begin{document}


\title[The finite embeddability property]{Semilinear substructural logics with the \\
finite embeddability property }

\author[S. M. Wang]{SanMin Wang}
\address{Faculty of Science\\Zhejiang Sci-Tech University\\ Hangzhou 310018\\ P.R. China}
\email{wangsanmin@hotmail.com}

\thanks{This research was supported by the National Foundation of Natural Sciences of China under Grant nos. 61379018 and 61662044 and 11571013}




\begin{abstract}
Three semilinear substructural logics ${\rm {\bf HpsUL}}_\omega ^\ast $, ${\rm {\bf UL}}_\omega $ and ${\rm {\bf IUL}}_\omega $ are constructed.
  Then the  completeness of  ${\rm {\bf UL}}_\omega $ and  ${\rm {\bf IUL}}_\omega$  with respect  to  classes of  finite \textbf{UL} and \textbf{IUL}-algebras, respectively, is  proved.   Algebraically,  non-integral ${\rm {\bf UL}}_\omega $ and ${\rm {\bf IUL}}_\omega $-algebras
   have the finite embeddability property, which  gives a characterization for finite \textbf{UL}
   and \textbf{IUL}-algebras.
\end{abstract}

\maketitle


\section{Introduction}\label{sec:intro}

The finite embeddability property (FEP),  or actually,   the finite model property (FMP), as shown in~\cite{W2}, fails for some known non-integral semilinear substructural logics including
Metcalfe and Montagna's uninorm
logic ${\rm {\bf UL}}$ and involutive uninorm logic ${\rm {\bf IUL}}$~\cite{MM}, and a
suitable extension ${\rm {\bf HpsUL}}^\ast $~\cite{W1} of Metcalfe, Olivetti and Gabbay's pseudo-uninorm logic ${\rm {\bf HpsUL}}$~\cite{MOG}. This shows that  ${\rm {\bf UL}}$, ${\rm {\bf IUL}}$ and ${\rm {\bf HpsUL}}^\ast $ are incomplete with respect to the corresponding classes of  finite algebras.

A natural problem is  whether we can construct logics which are complete with respect to finite ${\rm {\bf UL}}$, ${\rm{\bf IUL}}$ and ${\rm {\bf HpsUL}}^\ast $-algebras.  Algebraically, our motivation is how to characterize the variety generated by  its finite members when a class of algebras does not enjoy the FEP (or FMP).

 In this paper, we construct three schematic extensions ${\rm {\bf UL}}_\omega
$, ${\rm {\bf IUL}}_\omega $ and ${\rm {\bf HpsUL}}_\omega ^\ast $ by adding
one simple axiom
$$\mathrm{(FIN)}\,\,\, \vdash(\varphi\backslash e)\leftrightarrow ((\varphi \odot\varphi)\backslash e)$$
to ${\rm {\bf UL}}$, ${\rm {\bf IUL}}$ and ${\rm {\bf
HpsUL}}^\ast $,  respectively,  where $\varphi \leftrightarrow \psi$ is defined  to  be $ (\varphi \backslash \psi) \wedge (\psi \backslash \varphi)$.  Then we prove that  ${\rm {\bf UL}}_\omega $ and
${\rm {\bf IUL}}_\omega $  are complete with respect to classes of
 finite \textbf{UL} and \textbf{IUL}-algebras, respectively.  Algebraically,
 non-integral ${\rm {\bf UL}}_\omega $ and ${\rm {\bf IUL}}_\omega $-algebras
  have the finite embeddability property, which  gives a characterization
  for finite \textbf{UL} and \textbf{IUL}-algebras.

Classes of  ${\rm {\bf UL}}_\omega
$ and ${\rm {\bf IUL}}_\omega $-algebras are non-integral varieties which usually,
as pointed out in~\cite{H},  do not enjoy the FEP.
We prove the FEP for ${\rm {\bf UL}}_\omega $ and ${\rm {\bf
IUL}}_\omega $-algebras by Blok and Alten's construction~\cite{BA1,BA2}. But  in proving  the finiteness of Blok and Alten's construction in Lemma~\ref{lem:FMP3},  we have not  used Dickson¡¯s
lemma~\cite{BA1, CMM} or Higman¡¯s finite basis theorem~\cite{BA2, HH}  but used a specific property of  ${\rm {\bf UL}}_\omega $ and ${\rm {\bf IUL}}_\omega $-algebras which is given in Lemma~\ref{lem:p1}.

Since almost all proofs are done algebraically, as suggested by the referee, Hilbert-style systems for the logics under consideration have not been given. For details on \textbf{UL} and \textbf{IUL}, we refer to~\cite{MM}.  For  details on ${\rm {\bf HpsUL}}^\ast $, we refer to~\cite{MOG, W1}.
In addition,  we are unable to prove the FEP for ${\rm {\bf HpsUL}}_\omega ^\ast $ and left it as an open problem. In the paper,  $\mathbb{Z}_+ $ denote the set of positive integers,
$\mathbb{N}=\mathbb{Z}_+ \cup \{\,0\,\}$.

\section{${ \rm {\bf HpsUL}}^{\ast}_{\omega}$-algebras, ${\mathbf{UL}_{\omega}}$-algebras  and  ${\mathbf{IUL}_{\omega}}
$-algebras} \label{sec:aaa}

\begin{definition}~\cite{JT, MOG, TB}
 An ${\rm {\bf HpsUL}}$-algebra is a bounded semilinear
residuated lattice $\mathcal{A}=\langle A,\wedge,\vee,\cdot,\backslash,/, e,f,\bot,\top\rangle$
with universe $A$, binary operations $\wedge,\vee,\cdot,\backslash,/ $, and constants $e, f, \bot,\top $ such that:
\begin{enumerate}
\item[(i)] $\langle A,\wedge,\vee,\bot,\top\rangle$ is
a bounded lattice with top element $\top$ and
bottom element $\bot $;

\item[(ii)] $\langle A,\cdot,e\rangle$ is a monoid;

\item[(iii)] $\forall
x,y,z\in A,  x\cdot y\leqslant z$ iff $x\leqslant z/ y$
iff $y\leqslant x\backslash z$;

 \item[(iv)] $\forall x,y,u,v\in A,(\lambda_{u}((x\vee y)\backslash
x))\vee(\rho_{v}((x\vee y)\backslash y))=e$, where, for any $a,b\in A,$
 $\lambda_{a}(b)\coloneq(a \backslash
(b\cdot a) ) \wedge e, \rho_{a}(b)\coloneq((a\cdot b)/a) \wedge e.$
\end{enumerate}

\end{definition}

 We use the convention that $\cdot$ binds stronger than other binary operations and
 we shall often omit $\cdot$.  For example,  we will thus write $xy$ instead of $x\cdot y$.  We also define  $x^0=e $ and  $x^{n+1}= x^n $.

\begin{definition} \cite{MOG, MM, W1}
Let $\mathcal{A}=\langle A,\wedge,\vee,\cdot,\backslash ,/,
e, f, \bot,\top\rangle$ be an  \\
${\bf{HpsUL}}$-algebra. Then
\begin{enumerate}
\item[(i)]  ${\mathcal A}$ is an ${\bf{HpsUL}}$-chain if
it is linearly ordered.
\item[(ii)]    ${\mathcal A}$ is an ${\bf{HpsUL}}^{\ast}$-algebra  if the following weak commutativity (Wcm) holds for all $x, y \in A$:
$$xy\leq e \,\,\,\mathrm{implies}\,\,\,  yx\leq e.$$

\item[(iii)]   ${\mathcal A}$
is an  $\bf{UL} $-algebra if $xy=yx$ for all $x, y \in A$.

\item[(iv)]   ${\mathcal A}$
is an  $\bf{IUL} $-algebra if it is an  $\bf{UL} $-algebra such that  $\neg\neg x=x$ for all $x\in A$.
\item[(v)]   ${\mathcal A}$
is an   $\bf{HpsUL}^{\ast}_{\omega} $-algebra
(${\bf{UL}_{\omega}}$ or $
{\bf{IUL}_{\omega}}$-algebra) if it is an  $\bf{HpsUL}^{\ast} $-algebra ($\bf{UL}$ or $\bf{IUL}$-algebra) such
 that the following identity (Fin)
 $$x \backslash e=x^{2} \backslash e$$
holds for all $x\in A$.
\end{enumerate}
\end{definition}

\begin{theorem}  \cite{MOG, MM, TB}  \label{thm:P1}
Let ${\bf{L}}\in\{\,{\rm {\bf HpsUL}}^{\ast},{\mathbf{UL}},{\mathbf{IUL}},{\rm {\bf HpsUL}}^{\ast}_{\omega},{\mathbf{UL}_{\omega}},
{\mathbf{IUL}_{\omega}}\,\}$.  Then
\begin{enumerate}
\item[(i)] Each ${\rm {\bf L}}$-algebra
has a subdirect representation with ${\rm {\bf L}}$-chains;  \item[(ii)]
each finite ${\rm {\bf L}}$-algebra
has a subdirect representation with finitely many finite ${\rm {\bf L}}$-chains.
\end{enumerate}
\end{theorem}

\begin{lemma}\label{lem:p1}
Let $\mathcal{A}$ be an ${\rm {\bf HpsUL}}^\ast_{\omega}$-algebra.  Then
\begin{enumerate}
\item[(i)] $xy\leqslant e$ iff $xy^2\leqslant e$ for any $x,y\in A$;

\item[(ii)] $x_{1}^{k_{1}}\cdots x_{n}^{k_{n}}\leqslant e$ iff $x_{1}^{l_{1}}\cdots x_{n}^{l_{n}}\leqslant e$ for any  $x_1, \dots,  x_{n}\in A$, $k_{1}, \dots , k_{n}$, $l_{1}, \dots ,l_{n}\in \mathbb{Z}_+$.
\end{enumerate}
\end{lemma}
\begin{proof}
(i)  Let  $xy\leqslant e$ then $yx\leqslant e$ by (Wcm). Thus $x\leqslant y\backslash e$. Thus $x\leqslant y^{2}\backslash  e$ by (Fin). Hence $y^{2}x\leqslant e$. Therefore $x y^{2}\leqslant e$ by (Wcm).
The sufficiency part of (i) is proved in the same way.
(ii) is immediate from  (i).
\end{proof}

\begin{lemma}\label{lem:p2}
Let $\mathcal{A}$ be an ${\rm {\bf HpsUL}}^\ast_\omega $-chain. Then
\begin{enumerate}
\item[(i)] $st>u$  iff  $t>s\backslash u$ iff  $s>u/t$;
\item[(ii)] $su>tu$  implies $s>t$;
\item[(iii)] $stu=u$ implies $tu=u$.
\end{enumerate}
\end{lemma}

\begin{proof} (i)  and (ii) are clear. Only (iii) is proved as follows.  If $st\leqslant e$ then $tst\leqslant e$ and
$sts\leqslant e$ by Lemma~\ref{lem:p1} and (Wcm). Thus $tu=tstu\leqslant u$ and
$u=ststu\leqslant tu$. Hence $tu\leqslant u$ and $u\leqslant tu$.
Therefore $tu=u$. The case of $st>e$ is proved in the same
way.  Thus $tu=u$.
\end{proof}

\begin{lemma}\label{lem:c2a}
\begin{enumerate}
\item[(i)]  Each finite ${\rm {\bf HpsUL}}^\ast$-chain is an  ${\rm {\bf HpsUL}}_{\omega}^\ast$-chain;

\item[(ii)]   Each finite ${\rm {\bf HpsUL}}^\ast$-algebra is an  ${\rm {\bf HpsUL}}_{\omega}^\ast$-algebra.
 \end{enumerate}
\end{lemma}

\begin{proof}
\begin{enumerate}
\item[(i)] Let $\mathcal{A}$ be a finite ${\rm {\bf HpsUL}}^\ast$-chain.  We prove that $x\backslash e=x^2\backslash e$ for all $x$ in $A$.
Since $\mathcal{A}$ is finite, there is a positive integer
$n$ such that $x^n= x^{n - 1}$ for all $x\in A$.

 Suppose that $x\backslash e>x^2\backslash e$ then
$x^2(x\backslash e)>e$.  Let $z=x\backslash e$ then
$xz\leqslant e<x^2z$.  Thus $zx\leqslant e<zx^2$ by $\mbox{(Wcm)}$.

If $x^k=x^{k-1}$ and $zx\leqslant e<zx^2$ for any $k\geqslant 3$.  Then
$zx^{k-1}\leqslant x^{k-2}\leqslant zx^k$.  Thus $zx^{k-1}=x^{k-2}$ by
$x^k=x^{k-1}$.  Hence $x^{k-1}=x^{k-2}$ by Lemma~\ref{lem:p2} (iii).

 Since $x^n=x^{n-1}$ and $zx\leqslant e<zx^2$, then $x^{n-1}=x^{n-2}$, {\dots}, $x^2=x$ by repeatedly applying the property above. Thus $zx=zx^2$,  a
contradiction and hence $x\backslash e\leqslant x^2\backslash e$.  Similarly,  we can prove
that $x^2\backslash e\leqslant x\backslash e$.  Thus $x\backslash e=x^2\backslash  e$.

\item[(ii)]  follows from (i) and Theorem~\ref{thm:P1} (ii).
\end{enumerate}
\end{proof}

Clearly, Lemmas~\ref{lem:p1} $\sim$~\ref{lem:c2a}  hold for all ${\rm {\bf UL}}_\omega $ and ${\rm {\bf IUL}}_\omega $-algebras.

\section{Blok and Alten's construction for ${\rm {\bf HpsUL}}_\omega ^\ast
$, ${\rm {\bf UL}}_\omega $, ${\rm {\bf IUL}}_\omega $-algebras}
\label{sec:bac}

\begin{definition}
Given an ordered algebra $\mathcal{A}=\left\langle
{A,\left\langle {f_i^\mathcal{A} :i\in I} \right\rangle ,\leqslant
^\mathcal{A}} \right\rangle $ (of any type), with $\leqslant ^\mathcal{A}$ a (partial) order on $A$, and any non-empty subset $B\subseteq A$,
the partial subalgebra $\mathcal{B}$ of $\mathcal{A}$ with domain $B$ is the ordered partial algebra $\mathcal{B}=\left\langle {B,\left\langle
{f_i^\mathcal{B} :i\in I} \right\rangle ,\leqslant ^\mathcal{B}}
\right\rangle $, where $a\leqslant ^\mathcal{B}b$ iff $a\leqslant
^\mathcal{A}b$ for all $a,b\in B$, and for each $i\in I$, $f_i$ $ k$-ary,
$b_1 ,\dots ,b_k \in B$,
\[
f_i^\mathcal{B} (b_1 ,\dots ,b_k )=
\begin{cases}
 f_i^\mathcal{A} (b_1 ,\dots ,b_k )&\quad\text{if } f_i^\mathcal{A} (b_1 ,\dots ,b_k )\in B, \\
 \mbox{undefined}&\quad \text{if } f_i^\mathcal{A} (b_1 ,\dots ,b_k )\notin B.
  \end{cases}
  \]
\end{definition}

\begin{definition}
A partial embedding of an ordered partial algebra
$\mathcal{B}$ into an ordered algebra $\mathcal{A}$ is a 1-1 map $\iota
:B\to A$ such that (i) $a\leqslant ^\mathcal{B}b$ iff $\iota (a)\leqslant
^\mathcal{A}\iota (b)$ for all $a,b\in B$; (ii) $\iota (f_i^\mathcal{B} (b_1
,\dots ,b_k ))=f_i^\mathcal{A} (\iota (b_1 ),\dots ,\iota (b_k ))$ if
$f_i^\mathcal{B} (b_1 ,\dots ,b_k )$ is defined for some operation $f_i $
and $b_1 ,\dots ,b_k \in B$ where $f_i^\mathcal{A} $ denotes the realization
of $f_i $ in $\mathcal{A}$.
\end{definition}

\begin{definition}
A class ${\rm {\bf K}}$ of ordered algebras of the same type has
the finite embeddability property (FEP for short) if every finite partial
subalgebra $\mathcal{B}$ of any algebra $\mathcal{A}\in {\rm {\bf K}}$ can
be partially embedded into some finite member of~${\rm {\bf K}}$.
\end{definition}

\begin{lemma}\label{lem:Ksi}
 Let ${\rm {\bf K}}$ be a variety and ${\rm {\bf K}}_{si} $
be the class of all subdirectly irreducible members of ${\rm {\bf K}}$. Then
${\rm {\bf K}}$ has the FEP if ${\rm {\bf K}}_{si} $ has the FEP.
\end{lemma}

\begin{proof}
See~\cite[Lemma 20]{CMM}.
\end{proof}

\begin{definition}
Let $\mathcal{A}=\left\langle {A,\cdot ,\backslash
,/ ,\wedge ,\vee ,e,f,\bot ,\top } \right\rangle $ be an  ${\rm {\bf HpsUL}}_\omega ^\ast $-chain and
$\mathcal{B}=\left\langle {B,\cdot ,\backslash ,/ ,\wedge ,\vee ,e,f,\bot
,\top } \right\rangle $ be a partial subalgebra of $\mathcal{A}$ such that
$\left\{\, {e,f,\bot ,\top } \right\}\subseteq B$. Let
$\mathcal{M}=\left\langle {M,\cdot ,\wedge ,\vee ,e,f,\bot ,\top }
\right\rangle $ be the linearly ordered submonoid of \\
$\left\langle {A,\cdot
,\wedge ,\vee ,e,f,\bot ,\top } \right\rangle $ generated by $B$.

Let $a_{1}, \dots, a_{n} \in M$ and let $\delta _1, \dots, \delta _n \in
\{\,l,r\,\}$ ($l$ and $r$ stand for ``left" and ``right", respectively). We will
write ${\rm {\bf a}}^{\rm {\bf \delta }}$ to denote the sequence
$a_1^{\delta _1 } \dots a_n^{\delta _n } $, we will use $\varepsilon $ to
denote the empty sequence and we denote by$M^{l,r}$ the set of all possible
${\rm {\bf a}}^{\rm {\bf \delta }}$, that is, $$M^{l,r}=\{\,\,a_{1}^{\delta _1 }
\dots a_n^{\delta _n } \mid n<\omega ; a_1 ,\dots ,a_n \in M;\delta _1 ,\dots
,\delta _n \in \{\,\,l,r\,\,\}\,\,\}. $$
Clearly any two elements of $M^{l,r}$ can be concatenated to form a new
element of $M^{l,r}$. The sequence ${\rm {\bf a}}^{\rm {\bf \delta }}$ is to
be understood as a unary polynomial operating on $M$, defined inductively as
follows: For each $c\in M$, set $\varepsilon (c)=c$ and, for ${\rm {\bf
a}}^{\rm {\bf \delta }}\in M^{l,r}$ and $b\in M$, set ${\rm {\bf a}}^{\rm
{\bf \delta }}b^l(c)={\rm {\bf a}}^{\rm {\bf \delta }}(b\cdot c)$ and ${\rm
{\bf a}}^{\rm {\bf \delta }}b^r(c)={\rm {\bf a}}^{\rm {\bf \delta }}(c\cdot
b)$.
\end{definition}

For each ${\rm {\bf a}}^{\rm {\bf \delta }}\in M^{l,r}$ and $b\in B$, define
$$({\rm {\bf a}}^{\rm {\bf \delta }})^{-1}\left( b \right]=\left\{\, {c\in
M\mid {\rm {\bf a}}^{\rm {\bf \delta }}(c)\leqslant b} \right\},\left( b
\right]=\left\{\, {c\in M\mid c\leqslant b} \right\}, $$
$$\bar {D}=\left\{\, {({\rm
{\bf a}}^{\rm {\bf \delta }})^{-1}\left( b \right]\mid {\rm {\bf a}}^{\rm {\bf
\delta }}\in M^{l,r},b\in B} \right\},D=\left\{\,
{\bigcap \chi
\mid\chi \subseteq \bar {D}} \right\}.$$

For $X\subseteq M$, define $$C(X)=\bigcap \left\{\, {({\rm {\bf a}}^{\rm {\bf
\delta }})^{-1}\left( b \right]\in \bar {D}\mid X\subseteq ({\rm {\bf a}}^{\rm
{\bf \delta }})^{-1}\left( b \right]} \right\}.$$

For $X,Y\subseteq M$ and $X_i \subseteq M, i\in I$, define
$$XY=\left\{\,
{ab\mid a\in X,b\in Y} \right\}, Xa=X\left\{\, a \right\},X\cdot ^DY=C(XY), $$
$$X\backslash ^DY=\left\{\, {a\in M\mid Xa\subseteq Y} \right\},
Y/^DX=\left\{\,
{a\in M\mid aX\subseteq Y} \right\},$$
$\qquad\qquad\quad\bigvee_{i\in I}^D X_i =C(\bigcup _{i\in I} X_i),\bigwedge _{i\in I}^D X_i =\bigcap _{i\in I} X_i ,\sim X=X\backslash ^D\left( {f}
\right], $
$$\bot ^D=\left( {\bot } \right]=\left\{\, \bot \right\},\top
^D=\left( {\top } \right]=M,e^D=\left( {e} \right],f^D=\left( {f}
\right].$$

When $\mathcal{A}$ is an ${\rm {\bf HpsUL}}_\omega ^\ast $-chain, all
${\rm {\bf a}}^{\rm {\bf \delta }}\in M^{l,r}$ have the form
$a_1^{l } a_2^{r} $ by the associativity of ${\rm {\bf HpsUL}}_\omega ^\ast $. Then
$(a_1^{l } a_2^{r})^{-1}\left( b \right]=\left\{\,
{c\in M\mid a_1 ca_2 \leqslant b} \right\}.$

When $\mathcal{A}$ is an ${\rm {\bf UL}}_\omega $-chain, we need not
$M^{l,r}$  to define $({\rm {\bf a}}^{\rm {\bf \delta }})^{-1}\left( b \right]$,  and simplify it as
$\left({a\mapsto b} \right]=\left\{\, {c\in M \mid ac\leqslant b} \right\}$ for all $a\in M$, $b\in B$.

\begin{lemma} \label{lem:BCP1}
If $\mathcal{A}$ is an ${\rm {\bf HpsUL}}_\omega ^\ast $-chain. Then the following properties hold.
\begin{enumerate}
\item[(1)]  $\left\{\, {\bot ^D,\top ^D,t^D,f^D} \right\}\subseteq
\bar {D}$ and $C(X)=X$ for all $X\in D$;

\item[(2)]  $ X\subseteq C(X), C(X)\subseteq C(Y)$ if $X\subseteq Y$ and $C(C(X))=C(X)$ for all
$X,Y\subseteq M;$

\item[(3)]  $ (X\vee ^DY)\backslash ^DZ=(X\backslash ^DZ)\wedge ^D(Y\backslash ^DZ);$

\item[(4)]  If $X\subseteq M$ and $Y_i \subseteq M$ for $i\in I$, then $X\backslash^D(\bigcap_{i\in I} Y_i )=\bigcap _{i\in I} ( {X\backslash ^DY_i })$ and $(\bigcap _{i\in I} Y_i) /^{D} X=\bigcap _{i\in I} \left( {Y_i /^{D} X} \right)$;

\item[(5)]  If $X\subseteq M$ and $Y\in D$ then $X\backslash ^DY\in D$ and
 $Y/^DX\in D$ ;

\item[(6)]  $ X\cdot ^De^D=e^D\cdot ^DX=X,
\left( {X\cdot ^DY} \right)\cdot ^DZ=X\cdot
^D\left( {Y\cdot ^DZ} \right)=C(XYZ)$ for all $X,Y,Z\in D$ and, $ X\cdot ^DY\subseteq e^D$ iff $Y\cdot ^DX\subseteq e^D$  for all $X,Y\in D$ ;

\item[(7)]  $ X\cdot ^DY\subseteq Z$ iff $Y\subseteq X\backslash ^DZ$  iff $ X\subseteq Z /^D Y$ for all $X,Y,Z\in
D$;

\item[(8)]  $X\backslash ^D(Y\backslash ^DZ)=(Y\cdot ^DX)\backslash ^DZ$ for all $X,Y\subseteq M$ and $Z\in D$;

\item[(9)]  $e^D=(\lambda_{U}((X\vee ^DY)\backslash ^DX))\vee^D(\rho_{V}((X\vee ^DY)\backslash ^DY))$  for all
$X,Y,U,V\in D$.

\item[(10)]  $\sim\sim\sim X=\sim X$ for all $X\subseteq M$;

\item[(11)]   If $a,b\in B$ and $a\backslash b\in B$ then $\left({a\backslash b}\right]=\left( {a}
\right]\backslash ^D\left( {b} \right]$,
where, $(10)$ and $ (11)$ are valid if  $\mathcal{A}$ is an ${\rm {\bf UL}}_\omega$ $(\mathrm{or}\,\,{\rm {\bf IUL}}_\omega)$-chain.
\end{enumerate}
\end{lemma}

\begin{proof}
See~\cite[Section 5]{BA1}  and~\cite[Section 2]{BA2}.
\end{proof}

\begin{lemma}  \label{lem:BCP2}
Let $\mathcal{A}$ be an ${\rm {\bf HpsUL}}_\omega ^\ast $-chain.  Then $X\cdot ^DY\subseteq \left( e \right]$ iff $X\cdot
^DY\cdot ^DY\subseteq \left( e \right]$ for all $X,Y\in D$.
\end{lemma}

\begin{proof}  Let $X\cdot ^DY\subseteq \left( e \right]$. Then $C(XY)\subseteq
\left( e \right]$. Thus $XY\subseteq \left( e \right]$. Hence $xy\leqslant
e$ for all $x\in X$, $y\in Y$. Let $x\in X$, $y,y'\in Y$ then $xy\leqslant e$ and $xy'\leqslant e$. Thus $yx\leqslant e$ by
$\mbox{(Wcm)}$. Then $yxxy'\leqslant e$. Hence $y'yxx\leqslant e$ by
$\mbox{(Wcm)}$. Thus $y'yx\leqslant e$ by Lemma~\ref{lem:p1}.
Therefore $xy'y\leqslant e$ by $\mbox{(Wcm)}$. Thus $XYY\subseteq \left( e
\right]$. Then $C(XYY)\subseteq \left( e \right]$ by Lemma~\ref{lem:BCP1}(2) and
$C(\left( e \right])=\left( e \right]$. Hence $X\cdot ^DY\cdot ^DY\subseteq
\left( e \right]$ by Lemma~\ref{lem:BCP1}(6).

Let $X\cdot ^DY\cdot ^DY\subseteq \left( e \right]$. Then $C(XYY)\subseteq
\left( e \right]$ by Lemma~\ref{lem:BCP1}(6). Thus $XYY\subseteq \left( e \right]$. Let
$x\in X$, $y\in Y$ then $xyy\leqslant e$. Thus $xy\leqslant e$ by
Lemma~\ref{lem:p1}.  Hence $XY\subseteq \left( e \right]$ .
Therefore $C(XY)\subseteq \left( e \right]$ by Lemma~\ref{lem:BCP1}(2) and $C(\left( e
\right])=\left( e \right]$. Then $X\cdot ^DY\subseteq \left( e \right]$.
\end{proof}

\begin{lemma} \label{lem:BCP3}
\begin{enumerate}
\item[(i)] $\mathcal{D}=\left\langle {D,\cdot ^D,\backslash ^D,/
^D,\vee ^D,\wedge ^D,e^D,f^D,\bot ^D,\top ^D} \right\rangle $ is an ${\rm
{\bf HpsUL}}_\omega ^\ast $-algebra if $\mathcal{A}$ is an $_{ }{\rm {\bf
HpsUL}}_\omega ^\ast $-chain;

\item[(ii)]  $\mathcal{D}=\left\langle {D,\cdot ^D,\backslash
^D,\vee ^D,\wedge ^D,e^D,\bot ^D,\top ^D} \right\rangle $ is an ${\rm {\bf
UL}}_\omega$-algebra if $\mathcal{A}$ is an ${\rm {\bf UL}}_\omega$-chain;

\item[(iii)]  $\mathcal{D}=\left\langle {D,\vee ^D,\wedge ^D,\bot ^D,\top
^D} \right\rangle $ is a complete lattice;

\item[(iv)]  $(\bigwedge _{i\in I}^D X_i)\backslash
^DY=\bigvee _{i\in I}^D (X_i \backslash ^DY)$ and
$(\bigvee _{i\in I}^D X_i)\backslash ^DY=\bigwedge
_{i\in I}^D (X_i \backslash ^DY)$.
\end{enumerate}
\end{lemma}

\begin{proof}  (i) and (ii) are immediate from Lemma~\ref{lem:BCP1}(6),(7),  (9) and Lemma~\ref{lem:BCP2}.  (iii) is clear.  (iv) follows  from (i),  (ii) and (iii).
\end{proof}

\begin{lemma} \label{lem:BCP4}
Let $\mathcal{A}$ be a linearly ordered ${\rm {\bf
IUL}}_\omega $-algebra and $\mathcal{B}$ be a partial subalgebra of
$\mathcal{A}$ such that $\left\{\, {e,f,\bot ,\top } \right\}\subseteq B$ and
$\neg b\in B$ for all $b\in B$. Then
\begin{enumerate}
\item[(i)] $\left( b \right]=\sim\sim\left( b\right]$;

\item[(ii)]  $\left( {a\mapsto b} \right]=\sim\sim\left( {a\mapsto b} \right]$
for all $\left( {a\mapsto b} \right]\in \bar {D}$;

\item[(iii)]  $ X=\sim\sim X$ for all $X\in D$;

\item[(iv)]  $\mathcal{D}=\left\langle {D,\cdot ^D,\to ^D,\vee
^D,\wedge ^D,e^D,f^D,\bot ^D,\top ^D} \right\rangle $ is an ${\rm {\bf
IUL}}_\omega $-algebra.
\end{enumerate}
\end{lemma}

\begin{proof}
\begin{enumerate}
\item[(i)] Let $b\in B$ then $\sim\sim\left(b \right]=(\sim\left( b
\right])\backslash^D\left( f \right]=\left( {\left( b \right]\backslash^D\left( f
\right]} \right)\backslash^D\left( f \right]= \quad \left( {\neg b} \right]\backslash^D\left( f \right]= {\left( {\neg b\to f} \right]}=\left(b\right]$ by Lemma~\ref{lem:BCP1}(11) and $b,f,\neg b\in B$. Thus $\sim\sim\left(b \right]=\left(b\right]$.

\item[(ii)]  Let $\left( {a\mapsto b} \right]\in \bar {D}$. Then $\left( {a\mapsto b}
\right]=\{\,a\,\}\backslash ^D\left( b \right]=\{\,a\,\}\backslash ^D\sim\sim\left( b \right]=$\\
$\{\,a\,\}\backslash
^D\left( {\sim\left( b \right]\backslash ^D\left( f \right]} \right)=
(\sim\left( b \right]\cdot^D\{\,a\,\})\backslash ^D\left( f \right]=\sim\left( \sim\left( b \right]\cdot^D\{\,a\,\}\right)$ by (i) and Lemma~\ref{lem:BCP1}(8). Thus $\left( {a\mapsto b}
\right]=\sim\left( \sim\left( b \right]\cdot^D\{\,a\,\} \right)$. Hence  by
Lemma~\ref{lem:BCP1}(10), $\sim\sim\left( {a\mapsto b} \right]=\sim\sim\sim\left( \sim\left( b \right]\cdot^D\{\,a\,\}
\right)=\sim\left(  \sim\left( b \right]\cdot^D\{\,a\,\} \right)=\left( {a\mapsto b}
\right]$.

\item[(iii)]  Let $X\in D$. Then we can write $X=\bigcap _{i\in I} \left( {a_i \wedge
b_i } \right]=\bigwedge _{i\in I}^D \left( {a_i \mapsto b_i } \right]$. Then
\begin{flalign*}
\sim\sim X&=(X\backslash ^D\left( f \right])\backslash ^D\left( f \right]\\
&=((\bigwedge{_{i\in I}^{D} }\left( {a_i \mapsto b_i } \right])\backslash ^D\left( f \right])\backslash
^D\left( f \right]\\
&=(\bigvee{_{i\in I}^{D}} (\left( {a_i \mapsto b_i } \right]\backslash
^D\left( f \right]))\backslash ^D\left( f \right]\\
&=\bigwedge{ _{i\in I}^D} ((\left( {a_i
\mapsto b_i } \right]\backslash ^D\left( f \right])\backslash ^D\left( f \right])\\
&=\bigwedge{_{i\in I}^D} \left( {a_i \mapsto b_i } \right]=X.
 \end{flalign*}
 by (ii) and Lemma~\ref{lem:BCP3} (iv).

\item[(iv)]  is immediate from (iii) and Lemma~\ref{lem:BCP3} (ii).
\end{enumerate}
\end{proof}

\begin{lemma}\label{lem:BCP5}
The map $\iota :\mathcal{B}\to \mathcal{D}$, which sends
$a$ to $\left( {a} \right]=\left\{\, {x\in M:x\leqslant a} \right\}$ for $a\in
B$, is an partial embedding of the partial subalgebra $\mathcal{B}$ of
$\mathcal{A}$ into $\mathcal{D}$. Moreover, $\iota (e)=e^D$, $\iota (f)=f^D$,
$\iota (\bot )=\bot ^D$, $\iota (\top )=\top ^D$ and $\iota $ preserves all
meets and joins that exist in $\mathcal{B}$.
\end{lemma}
\begin{proof}  It is proved by a procedure similar to that of~\cite[Lemma 2.6]{BA2}.
\end{proof}

\section{Finite embeddability property and decidability }\label{sec:fmp}

In this section we show that $\mathbf{UL}_\omega $ and $\mathbf{IUL}_\omega $ have the finite
embeddability property and are hence decidable.

We sometimes write $p_1 \cdot \cdots
\cdot p_k $ by $\prod\limits_{i=1}^k {p_i } $ for simplicity. We denote
the $i$-th component of $\alpha =(m_{1} ,\dots ,m_k )\in \mathbb{N}^k$ by
$\alpha (i)$, i.e., $\alpha (i)=m_i $ for all $1\leqslant i\leqslant k$.

\begin{definition}
A subsequence index is a mapping $\sigma :\mathbb{Z}_+
\to \mathbb{Z}_+ $  such that  $n\leqslant \sigma (n)<\sigma (n+1)$ for all
$n$ in $\mathbb{Z}_+ $.
\end{definition}

\begin{remark}There is a correspondence between the set of subsequences
of a sequence and the set of subsequence indexes, i.e., (i) Let $\{\,\alpha _n
\,\}$ be a sequence and $\sigma $ be a subsequence index then $\{\,\alpha_{\sigma(n)} \,\}$ is a subsequence of $\{\,\alpha _n \,\}$;  (ii) There is a subsequence index $\sigma $ for each subsequence $\{\,\alpha _{n_l } \,\}$ of
$\{\,\alpha _n \,\}$ such that $\sigma (l)=n_l $ for all $l $ in
$\mathbb{Z}_+ $.
\end{remark}

\begin{definition}Let $k\in \mathbb{Z}_+ $ and $\{\,\alpha _n \,\}$ be a
sequence in $\mathbb{N}^k$. $\{\,\alpha _n \,\}$ is an $\Omega $-sequence if
$\{\,\alpha _n (i)\,\}$ is an infinite constant chain or an infinite strictly ascending
chain for all $1\leqslant i\leqslant k$. A subsequence index $\sigma $ is an
$\Omega $-subsequence index  of $\{\,\alpha _n \,\}$ if $\{\,\alpha _{\sigma (n)} \,\}$ is an $\Omega
$-subsequence of $\{\,\alpha _n \,\}$.
\end{definition}

\begin{lemma}\label{lem:FMP1}
\begin{enumerate}
\item[(i)] Let $\{\,\alpha _n \,\}$ be a sequence in $\mathbb{N}$ then
there exists a subsequence index $\sigma $ such that $\{\,\alpha _{\sigma(n)} \,\}$ is an $\Omega $-subsequence of $\{\,\alpha _n \,\}$;

\item[(ii)]  The composition $\sigma _1 \circ \sigma _2 $ of two subsequence indexes
$\sigma _1 $ and $\sigma _2 $ of $\{\,\alpha _n \,\}$ is a subsequence index of $\{\,\alpha _n \,\}$;

\item[(iii)]  $\sigma _2 \circ\sigma _1 $ is an $\Omega $-subsequence index  of $\{\,\alpha _n \,\}$ if $\sigma _2 $ is an $\Omega
$-subsequence index  of $\{\,\alpha _n \,\}$ and $\sigma _1 $ is a subsequence index of $\{\,\alpha _n \,\}$.
\end{enumerate}
\end{lemma}

\begin{proof}
\begin{enumerate}
\item[(i)] If $\{\,\alpha _n \,\}$ is bounded,  it contains an infinite constant
subsequence. Otherwise it contains an infinite
strictly ascending subsequence. Then it contains an $\Omega $-subsequence.

\item[(ii)] That is to say, the subsequence of any subsequence of $\{\,\alpha _n\,\}$ is a subsequence of  $\{\,\alpha _n \,\}$.

\item[(iii)] That is to say, the subsequence $\{\,\alpha _{\sigma _2 \circ \sigma _1(n)} \,\}$ of the $\Omega $-subsequence \\
    $\{\,\alpha _{\sigma _2 (n)} \,\}_{ }$ is an $\Omega $- subsequence.
\end{enumerate}
\end{proof}

\begin{lemma}\label{lem:FMP2}
Let $\{\,\alpha _n \,\}$ be a sequence in $\mathbb{N}^k$. Then
there exists a subsequence index $\sigma $ such that $\{\,\alpha _{\sigma(n)} \,\}$ is an $\Omega $-subsequence of $\{\,\alpha _n \,\}$.
\end{lemma}

\begin{proof}
Since $\{\,\alpha _{n}(1)\,\}$  is a  sequence in $\mathbb{N}$,  then, by Lemma~\ref{lem:FMP1} (i),  there exists a subsequence index $\sigma_1$ such that $\{\,\alpha _{\sigma _1(n) }(1)\,\}$ is an $\Omega $-subsequence of $\{\,\alpha _{n}(1)\,\}$.   Thus $\{\,\alpha _{\sigma _1(n) }\,\}$ is a subsequence of $\{\,\alpha _{n}\,\}$ such that $\{\,\alpha _{\sigma _1(n) }(1)\,\}$ is an $\Omega $-subsequence of $\{\,\alpha _{n}(1)\,\}$. Note that  $\{\,\alpha _{\sigma _1(n) }\,\}$ is also a sequence in $\mathbb{N}^k$.

 Similarly,   $\{\,\alpha _{\sigma _1(n) }(2)\,\}$  is a  sequence in $\mathbb{N}$,  then, by Lemma~\ref{lem:FMP1} (i),  there exists a subsequence index $\sigma_2$ such that $\{\,\alpha _{\sigma _1(\sigma_2(n)) }(2)\,\}$ is an $\Omega $-subsequence of $\{\,\alpha _{\sigma _1(n) }(2)\,\}$.
Since $\{\,\alpha _{\sigma _1(\sigma_2(n)) }(1)\,\}$ is a subsequence of
$\{\,\alpha _{\sigma _1(n) }(1)\,\}$, then it is also an $\Omega $-subsequence of $\{\,\alpha _{n}(1)\,\}$ by Lemma~\ref{lem:FMP1} (iii).

Sequentially, we construct subsequence indexes $\sigma _1
,\sigma _2 ,\dots ,\sigma _k $ such that
\[
\{\,\alpha _{\sigma _1(n) }(1)\,\}, \quad \{\,\alpha _{\sigma _1 \circ \sigma _2(n)} (2)\,\},\quad \{\,\alpha_{\sigma _1 \circ \dots \circ \sigma _k(n) } (k)\,\}
\]
are $\Omega
$-subsequences by Lemma~\ref{lem:FMP1} (i) and (iii). Let $\sigma =\sigma _1 \circ \cdots\circ \sigma _k $ then $\{\,\alpha _{\sigma(n)} \,\}$ is an $\Omega
$-subsequence of $\{\,\alpha _n \,\}$ by Lemma~\ref{lem:FMP1}(iii).
\end{proof}

\begin{definition} Let $B =\{\,p_1 ,\dots ,p_k \,\}$,
$\mathcal{M}=\{\,\prod\limits_{i=1}^k {p_i ^{\alpha (i)}}\mid\alpha \in
\mathbb{N}^k\,\}$. $\left( {a\mapsto b} \right]=\left\{\, {c\in
\mathcal{M}\mid ac\leqslant b} \right\}$ for all $a\in
\mathcal{M},b\in B$  are as Definition 3.5.
\end{definition}

\begin{lemma}\label{lem:FMP3}
For every $p\in B $, let $M\Rightarrow
p=\{\,\left( {m\mapsto p} \right]\mid m\in \mathcal{M}\,\}$ . Then (i) $M\Rightarrow p$ is
linearly ordered under set inclusion; (ii) $M\Rightarrow p$ is finite.
\end{lemma}

\begin{proof} (i) follows directly from the fact that these elements of $D$ are all down-sets of $A.$

(ii) Suppose that there is an infinite strictly  ascending sequence\\
$\left\{\, {\left(
{\prod\limits_{i=1}^k {p_i ^{\beta _n (i)}} \mapsto p} \right]} \right\}$ in
$M\Rightarrow p$ under set inclusion, i.e., $\beta _n \in \mathbb{N}^k$
and

$$
\left( {\prod\limits_{i=1}^k {p_i ^{\beta _n (i)}}\mapsto p}
\right]\subset \left( {\prod\limits_{i=1}^k {p_i ^{\beta _{n+1} (i)}}\mapsto p}
\right]
$$
for all $n\in \mathbb{Z}_+ $. By Lemma~\ref{lem:FMP2}, there is a subsequence
index $\tau $ such that $\{\,\beta _{\tau(n)} \,\}$ is an $\Omega
$-subsequence of $\{\,\beta _n \,\}$. Then
 $$\left( {\prod\limits_{i=1}^k {p_i
^{\beta _{\tau (n)} (i)}}\mapsto p} \right]\subset \left( {\prod\limits_{i=1}^k
{p_i ^{\beta _{\tau (n+1)} (i)}} \mapsto p} \right] $$
 for all $n$ in $\mathbb{Z}_+$, where  for each $1\leqslant i\leqslant k$,
\[
\beta _{\tau(1)} (i)=\beta _{\tau(2)} (i)=\cdots =\beta _{\tau (n)}(i)=\cdots
\]
or
\begin{equation}
\beta _{\tau(1)} (i)<\beta _{\tau (2)}(i)<\cdots <
\beta_{\tau(n)} (i)<\cdots.
\label{eq:1}
\end{equation}
Thus there is
a sequence $\{\,\alpha _n \,\}$ in $\mathbb{N}^k$ such that
$$\prod\limits_{i=1}^k {p_i ^{\alpha _n (i)}} \in \left(
{\prod\limits_{i=1}^k {p_i ^{\beta _{\tau (n+1)} (i)}} \mapsto p} \right]
 \,\,\mathrm{and}\,\,
\prod\limits_{i=1}^k {p_i ^{\alpha _n (i)}} \notin \left(
{\prod\limits_{i=1}^k {p_i ^{\beta _{\tau (n)} (i)}}\mapsto  p} \right]. $$
Thus for all $n$ in $\mathbb{Z}_+ $,
$$\prod\limits_{i=1}^k {p_i ^{\beta _{\tau (n+1)} (i)}} \prod\limits_{i=1}^k
{p_i ^{\alpha _n (i)}} \leqslant p<\prod\limits_{i=1}^k {p_i ^{\beta _{\tau
(n)} (i)}} \prod\limits_{i=1}^k {p_i ^{\alpha _n (i)}}. $$
By Lemma~\ref{lem:FMP2}, there is a subsequence index $\sigma $ such that
$\{\,\alpha _{\sigma(n)} \,\}$ is an $\Omega $-subsequence of $\{\,\alpha _n
\,\}$. Then  for each
$1\leqslant i\leqslant k$,
\[
\alpha _{\sigma(1)} (i)=\alpha _{\sigma(2)} (i)=\cdots
=\alpha _{\sigma(n)} (i)=\cdots\]
 or
 \begin{equation}
\alpha _{\sigma (1)} (i)<\alpha_{\sigma(2)} (i)<\cdots <
\alpha _{\sigma(n)} (i)<\cdots.
\label{eq:2}
\end{equation}

Then for all $n$ in $\mathbb{Z}_+ $,
\begin{equation}
\prod\limits_{i=1}^k {p_i ^{\beta _{\tau
(\sigma (n)+1)} (i)}} \prod\limits_{i=1}^k {p_i ^{\alpha _{\sigma (n)} (i)}} \leqslant
p<\prod\limits_{i=1}^k {p_i ^{\beta _{\tau (\sigma (n))} (i)}}
\prod\limits_{i=1}^k {p_i ^{\alpha _{\sigma (n)} (i)}}.
\label{eq:3}
\end{equation}

Then by letting $n=1,3$ in~\eqref{eq:3},
$$\prod\limits_{i=1}^k {p_i ^{\beta _{\tau (\sigma (1)+1)} (i)}}
\prod\limits_{i=1}^k {p_i ^{\alpha _{\sigma (1)} (i)}} \leqslant
p<\prod\limits_{i=1}^k {p_i ^{\beta _{\tau (\sigma (1))} (i)}}
\prod\limits_{i=1}^k {p_i ^{\alpha _{\sigma (1)} (i)}}, $$
$$\prod\limits_{i=1}^k {p_i ^{\beta _{\tau (\sigma (3)+1)} (i)}}
\prod\limits_{i=1}^k {p_i ^{\alpha _{\sigma (3)} (i)}} \leqslant
p<\prod\limits_{i=1}^k {p_i ^{\beta _{\tau (\sigma (3))} (i)}}
\prod\limits_{i=1}^k {p_i ^{\alpha _{\sigma (3)} (i)}} .$$
Thus
\[
\prod\limits_{i=1}^k {p_i ^{\beta _{\tau (\sigma (1)+1)} (i)}}
\prod\limits_{i=1}^k {p_i ^{\alpha _{\sigma (1)} (i)}} <\prod\limits_{i=1}^k
{p_i ^{\beta _{\tau (\sigma (3))} (i)}} \prod\limits_{i=1}^k {p_i ^{\alpha
_{\sigma (3)} (i)}}
\]
and
\begin{equation}
\prod\limits_{i=1}^k {p_i ^{\beta _{\tau (\sigma
(3)+1)} (i)}} \prod\limits_{i=1}^k {p_i ^{\alpha _{\sigma (3)} (i)}}
<\prod\limits_{i=1}^k {p_i ^{\beta _{\tau (\sigma (1))} (i)}}
\prod\limits_{i=1}^k {p_i ^{\alpha _{\sigma (1)} (i)}} .
\label{eq:4}
\end{equation}
Since $$\sigma
(1)< \sigma
(1)+1\leqslant \sigma (2)<\sigma (3) \,\,\,\mathrm{and }$$
$$\sigma (1)<\sigma (3)+1, $$ then
 by~\eqref{eq:1} and~\eqref{eq:2},  for all $1\leqslant i\leqslant k,$
\begin{equation}
\beta _{\tau (\sigma (3))} (i)-\beta _{\tau (\sigma (1)+1)}(i)\geqslant 0,
\label{eq:5}
\end{equation}
\begin{equation}
\beta _{\tau (\sigma (3)+1)} (i)-\beta _{\tau (\sigma (1))}
(i)\geqslant 0,
\label{eq:6}
\end{equation}
\begin{equation}
\alpha _{\sigma (3)} (i)-\alpha _{\sigma (1)}
(i)\geqslant 0.\label{eq:7}
\end{equation}
Hence
\[
e<\prod\limits_{i=1}^k {p_i ^{\beta _{\tau (\sigma (3))} (i)-\beta _{\tau
(\sigma (1)+1)} (i)}} \prod\limits_{i=1}^k {p_i ^{\alpha _{\sigma (3)}
(i)-\alpha _{\sigma (1)} (i)}}
\]
and
\begin{equation}
\prod\limits_{i=1}^k {p_i ^{\beta
_{\tau (\sigma (3)+1)} (i)-\beta _{\tau (\sigma (1))} (i)}}
\prod\limits_{i=1}^k {p_i ^{\alpha _{\sigma (3)} (i)-\alpha _{\sigma (1)}
(i)}} <e \label{eq:8}
\end{equation}

by~\eqref{eq:4} and Lemma~\ref{lem:p2}(ii).

On the other hand,   for all $1\leqslant i\leqslant k$,
\[
\beta _{\tau (\sigma (3))} (i)-\beta _{\tau (\sigma
(1)+1)} (i)>0 \]
iff
\begin{equation}\beta _{\tau (\sigma (3)+1)} (i)-\beta _{\tau
(\sigma (1))} (i)>0\label{eq:9}
\end{equation}
by~\eqref{eq:1}  and
$\sigma(1)<\sigma (1)+1<\sigma (3)<\sigma (3)+1.$\\
Then \[
\beta
_{\tau (\sigma (3))} (i)-\beta _{\tau (\sigma (1)+1)} (i)+\alpha _{\sigma
(3)} (i)-\alpha _{\sigma (1)} (i)>0\] iff \begin{equation}
\beta _{\tau (\sigma
(3)+1)} (i)-\beta _{\tau (\sigma (1))} (i)+\alpha _{\sigma (3)} (i)-\alpha
_{\sigma (1)} (i)>0.
\label{eq:10}
\end{equation}

The  necessity part of~\eqref{eq:10}  is proved as follows and,  the sufficiency part is by a similar procedure and omitted.\\
Let
$$\beta _{\tau (\sigma (3))} (i)-\beta _{\tau (\sigma (1)+1)} (i)+\alpha
_{\sigma (3)} (i)-\alpha _{\sigma (1)} (i)>0.$$ Then by~\eqref{eq:5} and~\eqref{eq:7},
$$\beta _{\tau (\sigma (3))} (i)-\beta _{\tau (\sigma (1)+1)} (i)>0$$ or $$\alpha
_{\sigma (3)} (i)-\alpha _{\sigma (1)} (i)>0.$$ \\
If $\beta
_{\tau (\sigma (3))} (i)-\beta _{\tau (\sigma (1)+1)} (i)>0$ then  by~\eqref{eq:9}, $$\beta
_{\tau (\sigma (3)+1)} (i)-\beta _{\tau (\sigma (1))} (i)>0$$  and
thus  by~\eqref{eq:7},  $$\beta _{\tau (\sigma (3)+1)} (i)-\beta _{\tau (\sigma (1))} (i)+\alpha
_{\sigma (3)} (i)-\alpha _{\sigma (1)} (i)>0.$$
 If $\alpha _{\sigma
(3)} (i)-\alpha _{\sigma (1)} (i)>0$ then by~\eqref{eq:6}, $$\beta _{\tau (\sigma (3)+1)}
(i)-\beta _{\tau (\sigma (1))} (i)+\alpha _{\sigma (3)} (i)-\alpha _{\sigma
(1)} (i)>0.$$

Hence $$\prod\limits_{i=1}^k {p_i ^{\beta _{\tau
(\sigma (3))} (i)-\beta _{\tau (\sigma (1)+1)} (i)}} \prod\limits_{i=1}^k
{p_i ^{\alpha _{\sigma (3)} (i)-\alpha _{\sigma (1)} (i)}} >e$$ iff $$\prod\limits_{i=1}^k {p_i ^{\beta _{\tau (\sigma (3)+1)} (i)-\beta
_{\tau (\sigma (1))} (i)}} \prod\limits_{i=1}^k {p_i ^{\alpha _{\sigma (3)}
(i)-\alpha _{\sigma (1)} (i)}} >e$$ by~\eqref{eq:10} and Lemma~\ref{lem:p1}(ii), which contradicts with~\eqref{eq:8}. Hence there is no infinite strictly ascending sequence in $M\Rightarrow p$.  Similarly, we can prove that there is no infinite strictly  descending
sequence in $M\Rightarrow p$. Thus $M\Rightarrow p$ is finite by (i).
\end{proof}

\begin{lemma}\label{lem:FMP4}
If $\mathcal{B}$ is a finite partial subalgebra
of $\mathcal{A}$ then the algebra $\mathcal{D}$ is finite.
\end{lemma}

\begin{proof} It is immediate from Lemma~\ref{lem:FMP3}.
\end{proof}

\begin{theorem}\label{thm:FMP1}
The varieties of $\mathbf{UL}_\omega $-algebras and $\mathbf{IUL}_\omega
$-algebras have the FEP.
\end{theorem}

\begin{proof} It is immediate from Lemmas~\ref{lem:BCP3}$\sim$~\ref{lem:BCP5}, Lemma~\ref{lem:FMP4}.
\end{proof}

\begin{corollary}
The universal theories of $\mathbf{UL}_\omega $-algebras and
$\mathbf{IUL}_\omega $-algebras are decidable.
\end{corollary}

\begin{theorem}  \label{thm:FMP2}
Let  $\mathbf{L}\in \{\,\mathbf{UL}, \mathbf{IUL}\,\}$ and $\mathbf{L}_{\omega}\in \{\,\mathbf{UL}_{\omega}, \mathbf{IUL}_{\omega}\,\}$.  For any formula $\varphi$ in $\mathbf{L}$, the following statements are equivalent:
\begin{enumerate}
\item[(i)] $\Gamma \vdash _{\rm {\bf L}_{\omega}} \varphi $;

\item[(ii)] $\Gamma\vDash _{\mathcal{A}} \varphi $ for every ${\rm {\bf L}_{\omega}}$-algebra $\mathcal{A}$;

\item[(iii)] $\Gamma\vDash _{\mathcal{A}} \varphi $ for every ${\rm {\bf L}_{\omega}}$-chain $\mathcal{A}$;

\item[(iv)]  $\Gamma\vDash _{\mathcal{A}} \varphi $ for every finite ${\rm {\bf L}_{\omega}}$-algebra $\mathcal{A}$;

\item[(v)]   $\Gamma\vDash _{\mathcal{A}} \varphi $ for every finite ${\rm {\bf L}}$-algebra $\mathcal{A}$.
\end{enumerate}
\end{theorem}

\begin{proof}
(i) is equivalent to (ii) by a canonical procedure.
(iii) implies (ii)  by Theorem~\ref{thm:P1} (i).   Clearly, (ii) implies (iv). Then (iii) implies (iv). (iv) implies (iii) by Theorem~\ref{thm:FMP1}.  (iv) is equivalent to (v) by Lemma~\ref{lem:c2a} (ii).
\end{proof}

Theorem~\ref{thm:FMP2} shows that,  as was expected,  axiomatic systems ${\rm {\bf UL}}_\omega$ and ${\rm {\bf IUL}}_\omega $  are complete with respect to finite ${\rm {\bf UL}}$ and ${\rm{\bf IUL}}$-algebras,  respectively.
In other words,   ${\mathbf{UL}_{\omega}}$  and ${\mathbf{IUL}_{\omega}}$  are logics for  the  classes of  finite ${\mathbf{UL}}$  and ${\mathbf{IUL}}-$algebras, respectively.

 \section{Concluding remarks}\label{sec:cr}

 The suitability of  Blok and Alten's
Construction for ${\rm {\bf UL}}_\omega $, ${\rm {\bf IUL}}_\omega $-algebras
mainly depends on that elements of the monoid $M$ generated by $\{\,p_1 ,\dots, p_k \,\}$
has the form $\prod\limits_{i=1}^k {p_i ^{\alpha (i)}}$.  It seems difficult to
extend the proof of Lemma~\ref{lem:FMP3}  to ${\rm {\bf HpsUL}}_\omega^{\ast} $-algebras.

\subsection*{Acknowledgements}
 I would like to thank the anonymous reviewer for carefully reading the first version of this article and many instructive suggestions. Especially,
 the current form of the axiom (Fin) is due to the reviewer and its old form is $(xy)\backslash e=(xy^{2}) \backslash e$.

\end{document}